\documentclass{article}

\usepackage{newa4}
\usepackage{latexsym}
\usepackage{amssymb}
\usepackage{amsmath} 
\usepackage{amsthm}
\usepackage{amsfonts}
\usepackage{mathabx}

\usepackage[T1]{fontenc}
\usepackage[utf8]{inputenc}

\usepackage{autobreak}
\usepackage{graphicx}

\newcommand{\NN}{\mathbb{N}}

\newcommand{\DD}{\mathbb{D}}

\newcommand{\RR}{\mathbb{R}}

\newcommand{\BBB}{\mathcal{B}}

\newcommand{\LLL}{\mathcal{L}}

\newcommand{\NNN}{\mathcal{N}}

\newcommand{\QQQ}{\mathcal{Q}}
\newcommand{\SSS}{\mathcal{S}}
\newcommand{\TTT}{\mathcal{T}}
\newcommand{\UUU}{\mathcal{U}}

\newcommand{\XXX}{\mathcal{X}}

\newcommand{\ball}[3]{\mathrm{B}_{#1}(#2,#3)}
\newcommand{\cball}[3]{\mathrm{B}^\le_{#1}(#2,#3)}

\newcommand{\set}[2]{\mbox{$\{\,#1 \mid #2 \,\}$}}
\newcommand{\fun}[3]{\mbox{$#1 \colon #2 \rightarrow #3$}}
\newcommand{\pfun}[3]{\mbox{$#1 \colon #2 \rightharpoonup #3$}}

\newcommand{\pair}[1]{\langle #1 \rangle}
\newcommand{\conv}{\mathclose\downarrow}

\def\pt{\mathop{\mathstrut\rm pt}}

\newtheorem{theorem}{Theorem}[section]
\newtheorem{lemma}[theorem]{Lemma}
\newtheorem{proposition}[theorem]{Proposition}
\newtheorem{corollary}[theorem]{Corollary}
\theoremstyle{definition}
\newtheorem{definition}[theorem]{Definition}

\numberwithin{equation}{section}

\newtheorem{example}[theorem]{Example}

                                      {\end{list}}

\def\dom{\mathop{\mathstrut\rm dom}}
\def\range{\mathop{\mathstrut\rm range}}

\def\int{\mathop{\mathstrut\rm int}\nolimits}
\def\ext{\mathop{\mathstrut\rm ext}\nolimits}
\def\cl{\mathop{\mathstrut\rm cl}\nolimits}

\title{Bi-topological spaces and the Continuity Problem\thanks{
\protect\includegraphics[width=1em]{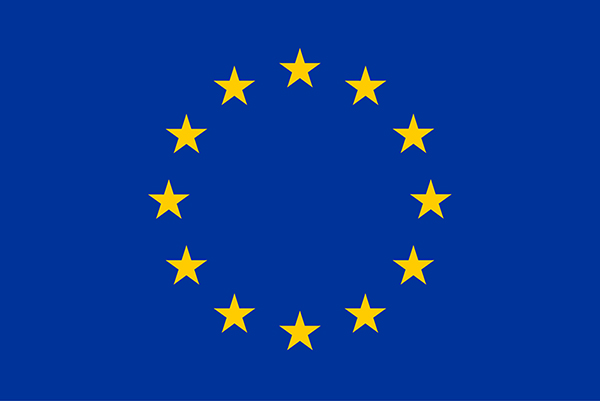}
This project has received funding from the European Union's Horizon 2020 research and innovation programme under the Marie Sk{\l}odowska-Curie grant agreement No 731143.}}

\author{Dieter Spreen\\
Department of Mathematics, University of Siegen\\
57068 Siegen, Germany}

\date{}

\begin{document}

 \maketitle

\begin{abstract}
The \emph{Continuity Problem} is the question whether effective operators are continuous, where an effective operator $F$ is a function on a space of constructively given objects $x$, defined by mapping construction instructions for $x$ to instructions for $F(x)$ in a computable way. In the present paper the problem is dealt with in a bi-topological setting. To this end the  topological setting developed by the author~\cite{sp} is extended to the bi-topological case. Under very natural conditions it is shown that an effective operator $F$ between bi-topological spaces $\TTT = (T, \tau, \sigma)$ and $\TTT' = (T', \tau', \sigma')$ is (effectively) continuous, if $\tau'$ is (effectively) regular with respect to $\sigma'$. A central requirement on $\TTT'$ is that bases of the neighbourhood filters of the points in $T'$ can computably be enumerated in a uniform way, not only with respect to topology $\tau'$, but also with respect to $\sigma'$. As follows from an example by Friedberg, the last condition is indispensable. Conversely, it is proved that (effectively) bi-continuous operators are effective. Prominent examples of bi-topological spaces are quasi-metric spaces. Under a very reasonable computability requirement on the quasi-metric it is shown that all effectivity assumptions made in the general results are satisfied in the quasi-metric case.
\end{abstract}

\section{Introduction}
\label{sec-intro}

Whereas in classical mathematics, one can prove the existence of a mathematical object without ``finding'' the object explicitly, by assuming its non-existence and then deriving a contradiction, in a constructive approach a \emph{construction} of the object has to be presented and it has to be shown that the object thus constructed has the property initially stated. But what is a construction? In the Russian school of constructive mathematics (founded by A.A.\ Markov, Jr.\ in the late 1940s-early 1950s) a construction is an \emph{algorithm}. For the latter notion many equivalent formalisations are known.

Let us consider the real numbers as an example. In classical real analysis, one way to define a real number is to establish an equivalence class of Cauchy sequences of rational numbers. In Markov-style constructive mathematics accordingly a real number is constructed by presenting a program $P$ that takes a positive integer $n$ and outputs a rational $a_{n}$, together with a program $Q$ that takes a positive integer $n$ and outputs a positive integer $N_{n}$ such that
\[
(\forall n)(\forall i, j \ge N_{n})\, |a_{i} -a_{j}| \le 2^{-n}.
\]
We can use $P$ and $Q$ together to compute as close a rational approximation as we like to the real number they represent.

We see that in Markov-style constructive analysis real numbers can be represented by pairs of programs.  Note that programs are just pieces of text. An \emph{effective operator}  $F$ from Markov reals to Markov reals (or, in general, between sets of constructive objects) is then a map $F$ for which there is a program translating programs $A_{x}$ representing objects $x$ in the domain of $F$ into programs $B_{F(x)}$ representing $F(x)$. 

The \emph{continuity problem} is the question whether effective operators are continuous.

For several classes of effectively presented spaces positive answers have been given: Classes of continuous domains (directed-complete partial orders) have been studied in \cite{ec,er,ms,st,wd} and of metric spaces in \cite{ce,ksl,mosch63,mosch64}. In \cite{sp} the present author studied the problem in a general topological setting. A natural condition was presented forcing effective operators between $T_{0}$ spaces to be continuous. The condition is satisfied in the cases just mentioned, i.e., continuous domains and metric spaces. Moreover, in the latter case the result could be extended to regular spaces.

In the present paper we will push the limit even further and show that effective operators between pairwise regular bi-topological spaces are bi-continuous. 

The generalisation to bi-topological spaces $(T, \tau, \sigma)$ is a natural step. As is well known from~\cite{be75,be76,be77,bs}, the continuity of an effective operator $\fun{F}{T}{T'}$ cannot be derived in a constructive setting without \emph{Markov's principle}, which says for a computable map $\alpha$,
\[
\neg \neg (\exists x\, \alpha\, x = 0) \Rightarrow (\exists x\, \alpha\, x = 0).
\]
Thus, for any $y \in T$, we have to assume that $F$ is not continuous at $y$. That is, for $V \in \tau'$ with $F(y) \in V$, we assume that
\[
(\forall U \in \tau) [y \in U \Rightarrow F[U] \not\subseteq V].
\]
In order to derive a contradiction, we need $F[U] \not\subseteq V$ to hold constructively, which means we can effectively find $z \in U$ so that $F(z) \notin V$. Regularity of $(T', \tau')$ allows to find some  $\widehat{V} \in \tau'$ with $F(y) \in \widehat{V} \subseteq \cl(\widehat{V}) \subseteq V$, where $\cl(\widehat{V})$ is the closure of $\widehat{V}$. Then $F[U]$ hits the exterior $\ext(\widehat{V})$ of $\widehat{V}$.  A central assumption for a space to be effectively presented is that for any (basic) open set its elements are effectively listable. So, by running through the listings for $U$ and $\widehat{V}$ we can effectively find a $z \in U$ with $F(z) \in \ext(\widehat{V})$. 

If we  assume that $\tau'$ is regular with respect to $\sigma'$ then we can proceed in a similar way. The only difference is that now we have that $\widehat{V} \in \sigma'$ with $\widehat{V} \supseteq T \setminus V$.

The paper is organised as follows: In Section~\ref{sec-bitop} bi-topological spaces and their regularity are discussed. As a prominent example,  quasi-pseudo-metric spaces are considered in Section~\ref{sec-qmet}. In Section~\ref{sec-effspa} effectivity requirements for topological spaces are recalled from~\cite{sp} and extended to the bi-topological case. Computable quasi-metric spaces are studied in Section~\ref{sec-compqm}. The general continuity results are derived in Section~\ref{sec-cont} and in Section~\ref{sec-fried}  the consequences of Friedberg's counterexample are discussed. The paper closes with a conclusion.

\section{Bi-topological spaces}\label{sec-bitop}

As is well known, on every topological space $(T, \tau)$ a canonical pre-order $\le_{\tau}$ is defined, called the \emph{specialisation pre-order}, which is a partial order exactly if $\tau$ satisfies the $T_{0}$ separation property. For $y, z \in T$, $y \le_{\tau} z$, if every  $O \in \tau$ containing $y$ also contains $z$. 

A \emph{bi-topological space} $\TTT = (T, \tau, \sigma)$ is a set $T$ with two topologies $\tau$ and $\sigma$ on $T$. The space $\TTT^{*} = (T, \sigma, \tau)$ is called its \emph{dual}.

\begin{example}\label{ex-realtop}
A typical example of a bi-topological space is the space $(\RR, \UUU, \LLL)$ with the \emph{upper} and the \emph{lower} topology, that is
\[
\UUU = \set{(a, \infty)}{a \in \RR}  \cup \{ \RR, \emptyset \} \quad\text{and}\quad \LLL = \set{(-\infty, a)}{a \in \RR} \cup \{ \RR, \emptyset \}.
\]
Both topologies do not have the $T_{2}$ separation property. However, they satisfy the $T_{0}$ condition. The specialisation order $\le_{\UUU}$ of $\UUU$ is the canonical order $\le$ on the real numbers, whereas $\le_{\LLL}$ is its converse $\ge$. 
The usual Euclidean topology on $\RR$ is the join $\UUU \vee \LLL$ of both topologies.
\end{example}

For two topologies $\tau$ and $\sigma$ the join $\tau \vee \sigma$ is generated by the intersections $U \cap V$ with $U \in \tau$ and $V \in \sigma$. It is also called the \emph{symmetrisation topology} of $\tau$ and $\sigma$, and $\TTT^{S} = (T, \tau \vee \sigma)$ is the \emph{symmetrisation space} of $\TTT$~\cite{ko95}.

\begin{definition}(\cite{ke})\label{dn-reg}
Let $\TTT = (T, \tau, \sigma)$ be a bi-topological space. 
\begin{enumerate}

\item\label{dn-reg-1} $\tau$ is \emph{regular with respect to} $\sigma$ if, for each point $x \in T$ and every $\tau$-closed subset $X$ with $x \notin X$, there are a $\tau$-open set $U$ and a $\sigma$-open set $V$ such that
\[
x \in U, \quad V\supseteq X, \quad\text{and}\quad U \cap V = \emptyset.
\]

\item\label{dn-reg-2} $\TTT$ is, or $\tau$ and $\sigma$ are, \emph{pairwise regular} if $\tau$ is regular with respect to $\sigma$ and vice versa.

\end{enumerate}
\end{definition}

\begin{example}\label{ex-pregreal}
$(\RR, \UUU, \LLL)$ is pairwise regular.
\end{example}

For a subset $X$ of $T$, let $\overline{X}$ denote its complement.

\begin{lemma}\label{lem-regbas}
Let $\TTT = (T, \tau, \sigma)$ be a bi-topological space with bases $\BBB_{\tau}$, $\BBB_{\sigma}$ of $\tau$ and $\sigma$, respectively. Then the following three statements are equivalent:
\begin{enumerate}

\item\label{lem-regbas-1}
$\tau$ is regular with respect to $\sigma$.

\item\label{lem-regbas-2}
For all $x \in T$ and  $B \in \BBB_{\tau}$ with $x \in B$, there is  $U \in \BBB_{\tau}$ so that 
\[
x \in U \subseteq B
\]
 and for all $z \in \overline{B}$, there is  $V_{z} \in \BBB_{\sigma}$ with 
 \[
 z \in V_{z} \quad\text{and}\quad V_{z} \cap U = \emptyset.
 \]

\item\label{lem-regbas-3}
For all $x \in T$ and  $B \in \BBB_{\tau}$ with $x \in B$, there are $U \in \BBB_{\tau}$ and $V \in \sigma$ so that
\[
x \in U \subseteq B, \quad \overline{B} \subseteq V, \quad\text{and}\quad U \cap V = \emptyset.
\]

\end{enumerate}
\end{lemma}
\begin{proof}
Assume (\ref{lem-regbas-1}). We will show (\ref{lem-regbas-2}). Let $x \in T$ and $B \in \BBB_{\tau}$. Then there are $\widehat{U} \in \tau$ and $\widehat{V} \in \sigma$ so that $x \in \widehat{U}$, $\widehat{U} \cap \widehat{V} = \emptyset$, and $\overline{B} \subseteq \widehat{V}$. Since $x \in \widehat{U} \cap B$, there is  $U \in \BBB_{\tau}$ with $x \in U \subseteq \widehat{U} \cap B$. Let $z \in \overline{B}$. Then $z \in \widehat{V}$. Hence, there exists $V_{z} \in \BBB_{\sigma}$ with $z \in V_{z} \subseteq \widehat{V}$. It follows that $U \cap V_{z} \subseteq \widehat{U} \cap \widehat{V} = \emptyset$.

Next, assume (\ref{lem-regbas-2}) and let  $x \in T$ and  $B \in \BBB_{\tau}$ with $x \in B$. Then there exists $U \in \BBB_{\tau}$ so that  $x \in U \subseteq B$. Moreover, there is $V_{z} \in \BBB_{\sigma}$ with $ z \in V_{z} \quad\text{and}\quad V_{z} \cap U = \emptyset$, for every $z \in \overline{B}$. Set $V = \bigcup_{z \in \overline{B}} V_{z}$. Then $V \in \sigma$ and $\overline{B} \subseteq V$. In addition, $U \cap V = \bigcup_{z \in \overline{B}} V_{z} = \emptyset$. That is, (\ref{lem-regbas-3}) holds.

Finally, assume (\ref{lem-regbas-3}) and let $x \in T$ with $x \notin X$, where $X \subseteq T$ is $\tau$-closed. Then $\overline{X} \in \tau$ and $x \in \overline{X}$. It follows that there is  $B \in \BBB_{\tau}$ with $x \in B \subseteq \overline{X}$. With (\ref{lem-regbas-3}) we therefore obtain that there are $U \in \BBB_{\tau}$ and $V \in \sigma$ with $x \in U \subseteq B$, $X = \overline{\overline{X}} \subseteq V$, and $U \cap V = \emptyset$, which shows that $\tau$ is regular with respect to $\sigma$.
\end{proof}

\section{Quasi-pseudo-metric spaces}\label{sec-qmet}

In this section we will consider an important subclass of bi-topological spaces.

\begin{definition}\label{dn-qmet}
A \emph{quasi-pseudo-metric} on $X$ is a function $\delta \colon X \times X \to [0,\infty)$ such that for all $x, y, z \in X$,
\begin{enumerate}

\item\label{dn-qmet-1}
$\delta(x, x) = 0$,

\item\label{dn-qmet-2}
$\delta(x, z) \le \delta(x, y) + \delta(y, z)$.

\end{enumerate}
If $\delta$ satisfies the stronger condition
\begin{enumerate}

\item[(\ref{dn-qmet-1}')] $\delta(x, y) = \delta(y, x) = 0$ only if $ x = y$,

\end{enumerate}
then $\delta$ is a \emph{quasi-metric}.
\end{definition}

Every quasi-pseudo-metric defines a pre-order $\sqsubseteq_{\delta}$ on $X$: $y \sqsubseteq_{\delta} z$, if $\delta(y, z) = 0$.

Let $\delta$ be a quasi-pseudo-metric on $X$. Then $\XXX = (X, \delta)$ is called \emph{quasi-pseudo-metric space}. Define $\delta^{c}$ by $\delta^{c}(x, y) = \delta(y, x)$. As is easily verified, $\delta^{c}$ is a quasi-pseudo-metric as well, called the \emph{conjugate} quasi-pseudo-metric of $\delta$. Moreover, $\XXX^{c} = (X, \delta^{c})$ is said to be the \emph{dual} of $\XXX$.

\begin{example}\label{ex-realmet}
For real numbers $y, z$ set 
\[
\delta_{L}(y, z) = \max \{ z - y, 0 \} \quad \text{and} \quad
\delta_{U}(y, z) = - \min  \{ z - y, 0 \}.
\]
Then $\delta_{L}$ and $\delta_{U}$ are both quasi-metrics such that $\delta_{L}$ is conjugate to $\delta_{U}$. Moreover, $y \sqsubseteq_{\delta_{L}} z$ exactly if $y \ge z$, and $y \sqsubseteq_{\delta_{U}} z$ exactly if $y \le z$.
\end{example}

For $x \in X$ and $r \in \RR_{+}$, let $\ball{\delta}{x}{r} = \set{y \in X}{\delta(x, y) < r}$ be the ball of radius $r$ about $x$, and $\cball{\delta}{x}{r} = \set{y \in X}{\delta(x, y) \le r}$. As in the classical case of metric spaces, the collection of all balls forms a basis for a topology $\tau_{\delta}$, called the \emph{quasi-pseudo-metric topology} of $\delta$. Its specialisation pre-order $\le_{\tau_{\delta}}$ coincides with the canonical pre-order $ \sqsubseteq_{\delta}$ associated with $\delta$. In Example~\ref{ex-realtop}, $\tau_{\delta_{L}} = \LLL$ and $\tau_{\delta_{U}} = \UUU$.

Similarly, $\delta^{c}$ determines a topology $\tau_{\delta^{c}}$ for $X$. Thus the natural topological structure associated with a quasi-pseudo-metric on a set $X$ is that of set $X$ with two topologies $\tau_{\delta}$ and $\tau_{\delta^{c}}$. By \cite[Prop.~4.1]{ke}, for every fixed $x \in X$, $\delta(x, y)$ is a $\tau_{\delta}$-upper semi-continuous and $\tau_{\delta^{c}}$-lower semi-continuous function of $y$. 

Set $\delta^{*}(x, y) = \max \{ \delta(x,y), \delta^{c}(x, y) \}$. Then $\delta^{*}$ is a pseudo-metric on $X$, called the \emph{symmetrisation} of $\delta$. The topology $\tau_{\delta^{*}}$ generated by $\delta^{*}$ is equivalent to the symmetrisation topology of $\tau_{\delta}$.

\begin{lemma}\label{lem-dbase}
Let $\QQQ$ be a dense subset of $X$ with respect to $\tau_{\delta^{*}}$ and $\DD = \set{2^{-e}}{e \in \NN}$. Then 
\[
\set{\ball{\delta}{u}{p}}{u \in \QQQ \wedge p \in \DD}
\]
is a basis of $\tau_{\delta}$.
\end{lemma}
\begin{proof}
By definition, $\tau_{\delta}$ is generated by the balls $\ball{\delta}{x}{r}$ with $x \in X$ and $r \in \RR_{+}$. Let $y \in \ball{\delta}{x}{r}$. Then $\delta(x, y) < r$. Let $a_{y} \in \DD$ with $a_{y} \le r - \delta(x, y)$ and choose $u_{y} \in \ball{\delta}{y}{a_{y}/2} \cap \ball{\delta^{c}}{y}{a_{y}/2} \cap \QQQ$. We show that
\begin{enumerate}

\item\label{dn-dbase-1}
$y \in \ball{\delta}{u_{y}}{a_{y}/2}$,

\item\label{dn-dbase-2}
$\ball{\delta}{u_{y}}{a_{y}/2} \subseteq \ball{\delta}{y}{a_{y}}$,

\item\label{dn-dbase-3}
$\ball{\delta}{y}{a_{y}} \subseteq \ball{\delta}{x}{r}$.

\end{enumerate}
Then 
\[
\ball{\delta}{x}{r} = \bigcup \set{\ball{\delta}{u_{y}}{a_{y}/2}}{y \in \ball{\delta}{x}{r} \wedge a_{y} \in \DD \wedge a_{y} \le r - \delta(x, y)},
\]
from which the statement follows.

(\ref{dn-dbase-1}) By choice, $u_{y} \in \ball{\delta^{c}}{y}{a_{y}/2}$. That is, $\delta(u_{y}, y) < a_{y}/2$.

(\ref{dn-dbase-2}) Let $z \in \ball{\delta}{u_{y}}{a_{y}/2}$, i.e., $\delta(u_{y}, z) < a_{y}/2$. Then
\[
\delta(y, z) \le \delta(y, u_{y}) + \delta(u_{y}, z) < a_{y}/2 + a_{y}/2 = a_{y}.
\]

(\ref{dn-dbase-3}) Let $z \in \ball{\delta}{y}{a_{y}}$. Then 
\[
\delta(x, z) \le \delta(x, y) + \delta(y, z) < \delta(x, y) + a_{y} \le r. \qedhere
\]
\end{proof}

Note that in (\ref{dn-dbase-3}) we even have that $\cball{\delta}{y}{a_{y}} \subseteq \ball{\delta}{x}{r}$.
The proposition that we are going to derive next is a special case of \cite[Prop.\ 4.2]{ke}. We will give a direct proof which we will refer to later on.

\begin{lemma}\label{lem-qmetreg}
Let $\QQQ$ be a dense subset of $X$ with respect to $\tau_{\delta^{*}}$. Then for every $\hat{u} \in \QQQ$, $r \in \RR_{+}$, $x \in \ball{\delta}{\hat{u}}{r}$, and $p \in \RR_{+}$ with $p < r - \delta(\hat{u}, x)$, thus there is  $u \in \QQQ$ so that
\begin{enumerate}

\item\label{lem-qmetreg-1}
$x \in \ball{\delta}{u}{p/2} \subseteq \ball{\delta}{\hat{u}}{r}$,

\end{enumerate}
and  for
\[
V = \bigcup \set{\ball{\delta^{c}}{v}{q/2}}{v \in \QQQ \wedge q \in \DD \wedge \delta(u, v) > p/2 \wedge q < \delta(u, v) - p/2},
\]
\begin{enumerate}
\stepcounter{enumi}

\item\label{lem-qmetreg-2}
$\overline{\ball{\delta}{\hat{u}}{r}} \subseteq V$, as well as

\item\label{lem-qmetreg-3}
$\ball{\delta}{u}{p/2} \cap V = \emptyset$.

\end{enumerate}
\end{lemma}
\begin{proof} 
Let $\hat{u} \in \QQQ$, $r \in \RR_{+}$, $x \in \ball{\delta}{\hat{u}}{r}$, and $p < r - \delta(\hat{u}, x)$. By our assumption on $\QQQ$, there is then some $u \in \QQQ$ so that
\[
\delta(\hat{u}, u) < p/2, \,\, \delta(u, \hat{u}) < p/2, \,\, x \in \ball{\delta}{u}{p/2}, \,\,\text{and}\,\, \cball{\delta}{u}{p/2} \subseteq \ball{\delta}{\hat{u}}{r},
\]
which proves Statement (\ref{lem-qmetreg-1}).

Next, let $z \in \overline{\ball{\delta}{\hat{u}}{r}}$. Then $z \in \overline{\cball{\delta}{u}{p/2}}$. That is, $\delta(u, z) > p/2$. Pick $q \in \DD$ with $q < \delta(u, z) - p/2$. Then $\set{y \in X}{\delta(u,y) > q + p/2}$, $\ball{\delta^{c}}{z}{q/2} \in \tau_{\delta^{c}}$ and $\ball{\delta}{z}{q/2} \in \tau_{\delta}$. Moreover, the intersection of the three sets is not empty. Hence, there is some 
\[
v \in \QQQ \cap \ball{\delta}{z}{q/2} \cap  \ball{\delta^{c}}{z}{q/2} \cap \set{y \in X}{\delta(u,y) > q + p/2}.
\]
It follows that $\delta(u, v) > q + p/2$. Moreover,
\begin{equation}\label{eq-qmetreg-1}
z \in \ball{\delta^{c}}{v}{q/2} \quad\text{and}\quad \ball{\delta^{c}}{v}{q/2} \subseteq \ball{\delta^{c}}{z}{q}.
\end{equation}

Now, assume that $y \in \cball{\delta}{u}{p/2} \cap \ball{\delta^{c}}{z}{q}$. Then $\delta(u, y) \le p/2$ and $\delta^{c}(z, y) < q$. Hence,
\begin{align*}
\delta(u, z)
& \le \delta(u, y) + \delta(y, z) \\
& = \delta(u, y) + \delta^{c}(z, y) \\
& < p/2 + q \\
&\le p/2 + (\delta(u, z) - p/2) = \delta(u, z).
\end{align*}
That is, $\delta(u, z) < \delta(u, z)$, a contradiction. Thus, 
\begin{equation}\label{eq-qmetreg-2}
\cball{\delta}{u}{p/2} \cap \ball{\delta^{c}}{z}{q} =\emptyset.
\end{equation} 
With (\ref{eq-qmetreg-1}) we therefore obtain that
\[
\ball{\delta}{u}{p/2} \cap \ball{\delta^{c}}{v}{q/2} = \emptyset.  
\]

As a further consequence of (\ref{eq-qmetreg-2}) it follows that 
\[
\ball{\delta^{c}}{v}{q/2} \subseteq \overline{ \cball{\delta}{u}{p/2}},
\]
which ensures that $\delta(u, v) > p/2$. Hence, we have that Statement (\ref{lem-qmetreg-2}) holds.

Finally, assume that $v \in \QQQ$ with $\delta(u, v) > p/2$ and $q < \delta(u, v) - p/2$. Then it follows as above that
\[
\ball{\delta}{u}{p/2} \cap \ball{\delta^{c}}{v}{q/2} = \emptyset,
\]
from which Statement (\ref{lem-qmetreg-3}) is a consequence.
\end{proof}

It results that $\tau_{\delta}$ is regular with respect to $\tau_{\delta^{c}}$. The next outcome follows by duality.

\begin{proposition}\label{pn-qmetreg}
Let $(X, \delta)$ be quasi-pseudo-metric space. Then the associated bi-topological space $(X, \tau_{\delta}, \tau_{\delta^{c}})$ is pairwise regular.
\end{proposition}

\section{Effective spaces}\label{sec-effspa}

In the remainder of this paper we will concentrate on spaces the elements of which are constructive. In our context this means that they are constructed by a computation process. Hence, they can be represented by a program. Note, however, that not every program needs to determine a construction process. A program is a finite piece of text and can therefore be coded into a natural number. 

Let $\langle\;,\;\rangle:\NN^{2}\rightarrow\NN$ be a
computable pairing function with corresponding  projections $\pi_1$ and $\pi_2$ such that $\pi_{i}(\langle a_{1},a_{2}\rangle )=a_{i}$. We inductively extend the pairing function to an $n$-tuple encoding for $n>2$ by $\pair{a_{1}, \ldots, a_{n}} = \pair{\pair{a_{1}, \ldots, a_{n-1}}, a_{n}}$.
Let $P^{(n)} (R^{(n)})$ denote the set of all $n$-ary partial (total) computable functions, and let $W_{i}$ be the domain of the $i$th partial computable function $\varphi_{i}$ with respect
to some G\"{o}del numbering $\varphi$. Moreover, we let $\varphi_{i}(a)\mathclose\downarrow$ mean that the computation of $\varphi_{i}(a)$  stops, and $\varphi_{i}(a)\mathclose\downarrow \in C$ that it stops with value in $C$.

Let $S$ be a non-empty set. A {\em (partial) numbering\/} $\nu$ of $S$ is a partial map $\pfun{\nu}{\NN}{S}$ (onto) with domain $\dom(\nu)$. 

\begin{definition}\label{dn-precom}
Let $\nu$ and $\kappa$ be numberings of the set $S$. 
\begin{enumerate}
\item 
$ \nu \le  \kappa$, read $\nu$ is \emph{reducible} to $\kappa$, if there is some function $g \in P^{(1)}$ with
$\dom(\nu) \subseteq \dom(g)$, $g(\dom(\nu)) \subseteq \dom(\kappa)$, and $\nu_{m} = \kappa_{g(m)}$, for all $m \in \dom(\nu)$.

\item
$\nu \equiv \kappa$, read $\nu$ is \emph{equivalent} to $\kappa$, if $\nu \le \kappa$ and $\kappa \le \nu$.

\end{enumerate}
\end{definition}

\begin{lemma}\label{lem-numbinf}
For the numberings $\nu, \kappa$ of $S$ and $i \in \dom(\nu), j \in \dom(\kappa)$ set 
\[
(\nu \ast \kappa)_{\pair{i,j}} = \nu_{i},  
\]
if $\nu_{i} = \kappa_{j}$. Otherwise, let $\nu \ast \kappa$ be undefined. Then $\nu \ast \kappa \le \nu, \kappa$.
\end{lemma} 

A subset $X$ of $S$ is {\em completely enumerable}, if there is a computably enumerable
set $W_n$ such that $\nu_i \in X$ if and only if $i \in W_n$, for all $i \in
\dom(\nu)$. Set $M_n = X$, for any such $n$ and $X$, and let $M_n$ be
undefined, otherwise. Then $M$ is a numbering of the class of all
completely enumerable subsets of $S$. $X$ is {\em enumerable}, if there is a computably enumerable set $A
\subseteq \dom(\nu)$ such that $X = \set{\nu_i}{i \in A}$. 

Now, let $\TTT = (T,\tau )$ be a countable topological
$T_{0}$-space with a countable basis $\BBB$. 
Moreover, let $B$ be a numbering of $\BBB$. In the applications we have in
mind the basic open sets can be described in a finite way. The
indexing $B$ is then obtained by an encoding of the finite
descriptions. Moreover, in these cases there is a canonical relation
between the (code numbers of the) descriptions which is stronger than
the usual set inclusion between the described sets. This relation is
computably enumerable, which in general is not true for set inclusion. One has to use
this stronger relation when introducing the notion of an effective space.
\begin{definition}
Let $\prec_B$ be a transitive binary relation on $\omega$. We say that:
\begin{enumerate}
\item $\prec_B$ is a {\em strong inclusion}, if for all $m$, $n\in
\dom(B)$, from $m\prec_B n$ it follows that $B_m \subseteq B_n$.

\item $\BBB$ is a {\em strong basis}, if $\prec_B$ is a strong
inclusion and for all $z \in T$ and $m$, $n \in \dom(B)$ with $z\in
B_m \cap B_n$ there is a number $a\in\dom(B)$ such that $z\in B_a$,
$a \prec_{B} m$ and $a \prec_{B} n$.
\end{enumerate}
\end{definition} 

If one considers basic open sets as vague descriptions, then strong 
inclusion relations can  be considered as `definite refinement'
relations~\cite{sm:86}.
Strong inclusion relations that satisfy much  stronger
requirements (in particular, they are extensional) appear very
naturally in the study of quasi-proximities~\cite{fl:82}. 
Moreover,
such relations have been used in Cz\'asz\'ar's approach to general
topology~\cite{cz:63} and in Smyth's work on topological foundations
of programming language semantics~\cite{sm:88,sm:94}. Compared with
the conditions used in these papers, the above requirements seem to
be  rather weak, but as we go along, we shall meet a further
requirement, and it is this condition which in applications prevents
us from choosing $\prec_B$ to be (the index set of) ordinary set
inclusion. For what follows we assume that $\prec_B$ is a strong
inclusion with respect to which $\BBB$ is a strong basis.
\begin{definition}
Let ${\TTT} = (T, \tau)$ be a countable topological $T_0$-space with
a countable basis $\BBB$, and let $x$ and $B$ be numberings of
$T$ and $\BBB$, respectively. We say that:
\begin{enumerate}
\item x is {\em computable (with respect to $\tau$)\/} if there is some computably enumerable set $L$ such that
for all $i\in\dom(x)$ and $n\in\dom(B)$, $\langle i,n\rangle \in L$ 
if and only if $x_{i}\in B_{n}$.

\item $\TTT$ is {\em effective}, if $B$ is a total indexing
and the property of being a strong basis holds effectively, which
means  that there exists a function $sb\in P^{(3)}$ such that for $i\in
\dom(x)$  and $m$, $n\in \NN$  with $x_{i}\in B_{m}\cap B_{n}$,
$sb(i,m,n)\mathclose\downarrow$, $x_{i}\in B_{sb(i,m,n)}$, 
$sb(i,m,n)\prec_B m$, and $sb(i,m,n)\prec_B n$.

\end{enumerate}
\end{definition}
Thus, $x$ is computable if and only if all basic open sets $B_n$ are
completely enumerable, uniformly in $n$.

Note that very often the totality of
$B$ can easily be achieved, if the space is {\em computably
separable}, which means that it has a dense enumerable subset, called
its \emph{dense base}. The  following lemma
presents a natural sufficient condition for a space to be effective.
\begin{lemma}[\cite{sp}]\label{T_eff}
Let $x$ be computable, $B$ be total, and $\prec_B$ be computably enumerable.
Then $\TTT$ is effective.
\end{lemma}

If $x$ is computable, a strong basis of basic open sets can effectively
be enumerated for each neighbourhood filter. As shown in \cite[Lemma~2.9]{sp}, for effective  spaces this can be done in a normed way.
\begin{definition}
An enumeration $( B_{f(a)} )_{a\in\NN}$  with
$f:\NN\rightarrow\NN$ such that $\range(f)\subseteq \dom(B)$ is 
said to be {\em normed\/} if $f$ is decreasing with respect to
$\prec_B$. If $f$ is computable, it is also called {\em computable\/} and
any G\"{o}del number of $f$ is said to be an {\em index\/} of it.
\end{definition}
In case $( B_{f(a)})$ enumerates a strong basis of the neighbourhood
filter of some point, we say it {\em converges\/} to that point. 

In what follows, we want not only to be able to generate normed
computable enumerations of  basic open sets that converge to a given
point, but conversely, we need also to be able to pass  effectively
from such enumerations to the point they converge to.
\begin{definition}
Let $x$ be a numbering of $T$. We say that:
\begin{enumerate}
\item $x$ {\em allows effective limit passing (with respect to $\tau$)\/} if there is a function
$pt\in P^{(1)}$ such that, if $m$ is an index of a  normed computable
enumeration of basic open sets which converges to some point
$y\in T$, then $pt(m)\mathclose\downarrow \in \dom(x)$ and $x_{pt(m)} =
y$.

\item $x$ is {\em acceptable (with respect to $\tau$)\/} if it allows effective limit passing
and is computable (both with respect to $\tau$).
\end{enumerate}
\end{definition}

If $x$ is computable, each neighbourhood filter ${\NNN}(y)$ has
a completely enumerable strong basis of basic open sets, namely the set
of all $B_a$ with $y \in B_a$.  As we shall see now, $T$ has an
 acceptable numbering, if, conversely, ${\NNN}(y)$ has an
enumerable strong basis of basic open sets, for all $y \in T$, and
$\prec_B$ is computably enumerable. By Lemma \ref{T_eff} it follows that $\TTT$ is
effective with respect to this numbering.
\begin{proposition}[\cite{sp}]\label{ex_numbering}
Let $\TTT$ be such that the neighbourhood filter of each point
has an enumerable strong basis of basic open sets. Moreover, let
$\prec_B$ be computably enumerable. Then $T$ has an acceptable
numbering with respect to which $\TTT$ is effective.
\end{proposition}

\begin{example}\label{ex-realcomp}
A real number $y$ is called \emph{right-computable} if the set 
\[
\set{\pair{a, b, c, e}}{x < (a - b)\cdot 2^{c-e}}
\]
 is computably enumerable. Similarly,  $y$ is \emph{left-computable} if the set 
 \[
 \set{\pair{a, b, c, e}}{(a - b)\cdot 2^{c-e} < x}
 \]
 is computably enumerable. If $y$ is both left- and right-computable, it is said to be \emph{computable}. Let $\RR_{c}$ denote the set of all computable real numbers. Obviously, $\RR_{c}$ is countable. We denote the restriction of topologies $\LLL$ and $\UUU$ to $\RR_{c}$, respectively, by $\LLL$ and $\UUU$ as well. Define
\[
B^{L}_{\pair{a, b, c, e}} = (- \infty, (a - b) \cdot 2^{c - e}) \quad \text{and} \quad B^{U}_{{\pair{a, b, c, e}}} = ((a - b) \cdot 2^{c - e}, \infty).
\]
Moreover, let  $\pair{a,b,c,e} \prec_{B^{L}} \pair{a',b',c',e'}$ if 
\[
(a - b) \cdot 2^{c-e} < (a' - b') \cdot 2^{c'-e'},
\]
 and $\pair{a,b,c,e} \prec_{B^{U}} \pair{a',b',c',e'}$ if 
 \[
 (a - b) \cdot 2^{c-e} > (a' - b') \cdot 2^{c'-e'}.
 \]
  Then $ \prec_{B^{L}}$ and $ \prec_{B^{U}}$ are computably enumerable, and $\BBB^{\LLL} = \range(B^{L})$ as well as $\BBB^{\UUU} = \range(B^{U})$ are strong bases of $\LLL$ and $\UUU$, respectively. It follows with respect to both topologies that the neighbourhood filter of each point $y \in \RR_{c}$ has an enumerable  strong basis of basic open sets. Hence, $T$ has numberings $x^{L}$ and $x^{U}$ so that $x^{L}$ is acceptable  with respect to $\LLL$ and  $x^{U}$ is acceptable  with respect to $\UUU$. In addition, both spaces, $(\RR_{c}, \LLL)$ and $(\RR_{c}, \UUU)$ are effective.
\end{example}

\begin{proposition}[\cite{sp}]\label{enum_spec}
Let $\TTT$ be effective and $x$ be acceptable. Then each completely
enumerable subset of $T$ is upwards closed under the specialization
order.
\end{proposition}

By definition each open set is the union of certain basic open
sets. In the context of  effective topology one is only interested in
enumerable unions. We call an open set $O\in\tau$  {\em Lacombe\/} set,
if there is an computably enumerable set $A\subseteq \dom(B)$ such that $O =
\bigcup\{\,B_{a}\mid  a\in A \,\}$. Set
$L^{\tau}_{n} = \bigcup\{\,B_{a}\mid  a\in W_{n} \,\}$, if
$W_{n}\subseteq \dom(B)$, and  let $L^{\tau}_{n}$ be undefined,
otherwise. Then $L^{\tau}$ is a numbering of the Lacombe sets of
$\tau$. Obviously, $B\le L^{\tau}$.

The following technical result will be needed later on.
\begin{lemma}[\cite{sp}]\label{intersec-DB}
Let $\TTT$ be effective and computably separable with dense base
$\QQQ$. Moreover, let $x$ be acceptable. Then, for any completely
enumerable subset $X$ of $T$ and any basic open set $B_n$, if $B_n$
intersects $X$, then it also intersects $X \cap \QQQ$.
\end{lemma}

\begin{definition}\label{dn-bieff}
Let $\TTT = (T, \tau, \sigma)$ be a countable bi-topological space such that $\tau$ and $\sigma$ have countable bases $\BBB^{\tau}$ and $\BBB^{\sigma}$, respectively, and satisfy the $T_{0}$ separation condition. Moreover, let $x$, $B^{\tau}$, and $B^{\sigma}$ be indexings of $T$, $\BBB^{\tau}$, and $\BBB^{\sigma}$, respectively, as well as $\prec_{\tau}$ and $\prec_{\sigma}$ be strong inclusions with respect to which $\BBB^{\tau}$ and $\BBB^{\sigma}$, respectively, are strong bases. We say that
\begin{enumerate}

\item\label{dn-bieff-1}
$x$ \emph{bi-computable} if $x$ is computable with respect to both topologies, $\tau$ and $\sigma$.

\item\label{dn-bieff-1a}
$x$ \emph{allows effective bi-limit passing} if there is some function $\pt \in P^{(1)}$ such that, if $m_{\tau}$ and $m_{\sigma}$ are indices of normed computable enumerations of basic open sets of $\tau$ and $\sigma$, respectively, both converging to the same point $y \in T$, then $\pt(\pair{m_{\tau}, m_{\sigma}})\conv \in \dom(x)$ and $x_{\pt(\pair{m_{\tau},m_{\sigma}})} = y$.

\item\label{dn-bieff-1b}
$x$ is \emph{bi-acceptable} if it allows effective bi-limit passing and is bi-computable.

\item\label{dn-bieff-2}
$\TTT$ \emph{effective} if space $(T, \tau)$ is effective and numbering $B^{\sigma}$ is total.

\item\label{dn-bieff-3}
$\TTT$ \emph{bi-effective} if both spaces, $(T, \tau)$ and $(T, \sigma)$, are effective.

\end{enumerate}
\end{definition}

As seen above, $(\RR_{c}, \UUU, \LLL)$ and $(\RR_{c}, \LLL, \UUU)$ are both bi-effective.

\begin{proposition}\label{pn-bicom}
Let $\TTT = (T, \tau, \sigma)$ be a countable bi-topological space such that $\tau$ and $\sigma$ have countable bases $\BBB^{\tau}$ and $\BBB^{\sigma}$, respectively, and satisfy the $T_{0}$ separation condition. Moreover, let $B^{\tau}$ and $B^{\sigma}$ be indexings of $\BBB^{\tau}$ and $\BBB^{\sigma}$, respectively, and $\prec_{\tau}$ and $\prec_{\sigma}$ be strong inclusions with respect to which $\BBB^{\tau}$ and $\BBB^{\sigma}$, respectively, are strong bases. Then the following statements hold:
\begin{enumerate}

\item\label{pn-bicom-1} 
For $m, m' \in \dom(B^{\tau})$ and  $n, n' \in \dom(B^{\sigma})$,  set 
\begin{gather*}
B^{\tau \vee \sigma}_{\pair{m,n}} = B^{\tau}_{m} \cap B^{\sigma}_{n} 
\intertext{and}
\pair{m,n} \prec_{B^{\tau \vee \sigma}} \pair{m', n'} \Longleftrightarrow m \prec_{\tau} m' \wedge n \prec_{\sigma} n'.
\end{gather*}
 Then $ \prec_{B^{\tau \vee \sigma}}$ is a strong inclusion and $\range(B^{\tau \vee \sigma})$ a strong basis of $\tau \vee \sigma$.

\item\label{pn-bicom-2} 
Let $x$ be a numbering of $T$. Then $x$ is bi-computable if and only if $x$ is computable with respect to $\tau \vee \sigma$.

\item\label{pn-bicom-3}
Let $x^{\tau}, x^{\sigma}$ be numberings of $T$ that are computable with respect to $\tau$ and $\sigma$, respectively. Then $x^{\tau} \ast x^{\sigma}$ is bi-computable.

\item\label{pn-bicom-4}
Let $x^{\tau}, x^{\sigma}$ be numberings of $T$ that allow effective limit passing with respect to $\tau$ and $\sigma$, respectively. Then $x^{\tau} \ast x^{\sigma}$ allows effective bi-limit passing.

\end{enumerate}
\end{proposition}
\begin{proof}
(\ref{pn-bicom-1}) The first assertion is obvious. 

(\ref{pn-bicom-2})Assume that $x$ is bi-computable and let this be witnessed by computably enumerable sets $L^{\tau}$ and $L^{\sigma}$. Define
\[
L = \set{\pair{i, \pair{m,n}}}{\pair{i, m} \in L^{\tau} \wedge \pair{i, n} \in L^{\sigma}}.
\]
Then $L$ is computably enumerable and we have for $i \in \dom(x)$ and $\pair{m, n} \in \dom(B^{\tau \vee \sigma})$,
\begin{align*}
x_{i} \in B^{\tau \vee \sigma}  
&\Longleftrightarrow x_{i} \in B^{\tau} \wedge x_{i} \in B^{\sigma} \\
&\Longleftrightarrow \pair{i, m} \in L^{\tau} \wedge \pair{i, n} \in L^{\sigma} \\
&\Longleftrightarrow \pair{i, \pair{m, n}} \in L.
\end{align*}

Next, conversely, let  $L \subseteq \NN$ be some computably enumerable set  witnessing that $x$ is computable with respect to $\tau \vee \sigma$. Set
\begin{align*}
&L^{\tau} = \set{\pair{i, m}}{(\exists n)\, \pair{i, \pair{m, n}} \in L}, \\
&L^{\sigma} =  \set{\pair{i, n}}{(\exists m)\, \pair{i, \pair{m, n}} \in L}.
\end{align*}
Then $L^{\tau}$ and $L^{\sigma}$ are both computably enumerable and for $i \in \dom(x)$ and $m \in \dom(B^{\tau})$,
\begin{align*}
x_{i} \in B^{\tau}_{m}
&\Longleftrightarrow x_{i} \in B^{\tau}_{m} \wedge x_{i } \in T \\
&\Longleftrightarrow (\exists n)\, (x_{i} \in B^{\tau}_{m} \wedge x_{i} \in B^{\sigma}_{n}) \\
&\Longleftrightarrow (\exists n)\, x_{i} \in B^{\tau \vee \sigma}_{\pair{m ,n}} \\
&\Longleftrightarrow (\exists n)\, \pair{i, \pair{m, n}} \in L \\
&\Longleftrightarrow \pair{i, m} \in L^{\tau}.
\end{align*}
Similarly, it follows for $i \in \dom(x)$ and $n \in \dom(B^{\sigma})$ that $x_{i} \in B^{\sigma}_{n}$, exactly if $\pair{i, n} \in L^{\sigma}$.

(\ref{pn-bicom-3}) Let $L^{\tau} \subseteq \NN$ witness that $x^{\tau}$ is computable with respect to $\tau$ and set $\hat{L}^{\tau} = \{\, \pair{\pair{i,j},n} \mid \pair{i,n} \in L^{\tau} \wedge j \in \NN \,\}$. Then $\hat{L}^{\tau}$ is computably enumerable and it follows for $i \in \dom(x^{\tau})$, $j \in \dom(x^{\sigma})$ and $n \in \dom(B^{\tau})$ that
\[
(x^{\tau} \ast x^{\sigma})_{\pair{i,j}} \in B^{\tau}_{n} \Longleftrightarrow x^{\tau}_{i} \in B^{\tau}_{n} \Longleftrightarrow \pair{i,n} \in L^{\tau} \Longleftrightarrow \pair{\pair{i,j},n} \in \hat{L}^{\tau}.
\]
Computability with respect to $\sigma$ follows analogously.

(\ref{pn-bicom-4}) Let $\pt_{\tau}, \pt_{\sigma} \in P^{(1)}$ witness that $x^{\tau}$ and $x^{\sigma}$, respectively, allow effective limit passing. Then $\pt(\pair{m_{1},m_{2}}) = \pair{\pt_{\tau}(m_{1}), \pt_{\sigma}(m_{2})}$ witnesses that $x^{\tau} \ast x^{\sigma}$ allows effective bi-limit passing.
\end{proof}

In case $\TTT$ is effective, topology $\tau$ is the major topology. Approximations, e.g., are done with respect to $\tau$. Topology $\sigma$ is only an auxiliary topology. 

\begin{definition}\label{dn-effreg}
Let $x$, $B^{\tau}$, and $B^{\sigma}$ be indexings of $T$, $\BBB^{\tau}$, and $\BBB^{\sigma}$, respectively. 
\begin{enumerate}

\item
$\tau$ is \emph{effectively regular with respect to $\sigma$}, if there are functions $s, t \in P^{(2)}$ such that for all $i \in \dom(x)$ and $m \in \dom(B^{\tau})$ with $x_{i} \in B^{\tau}_{m}$, 
\begin{enumerate}

\item\label{dn-effreg-1}
$s(i, m)\conv \in \dom(B^{\tau})$ and $t(i,m) \in \dom(L^{\sigma})$,

\item\label{dn-effreg-2}
$x_{i} \in B^{\tau}_{s(i,m)} \subseteq B^{\tau}_{m}$,

\item\label{dn-effreg-3}
$\overline{B^{\tau}_{m}} \subseteq L^{\sigma}_{t(i,m)}$, and

\item\label{dn-effreg-4}
$B^{\tau}_{s(i,m)} \cap L^{\sigma}_{t(i,m)} = \emptyset$.

\end{enumerate}

\item 
$\TTT$ is, or $\tau$ and $\sigma$ are, \emph{effectively pairwise regular} if $\tau$ is effectively regular with respect to $\sigma$ and vice versa.

\end{enumerate}
\end{definition}

Assume that we are given a point $x_{i} \in T$ and a basic open set $B^{\tau}_{m}$ with $x_{i} \in B^{\tau}_{m}$ and we need to find some point $z$ in the complement $\overline{B^{\tau}_{m}}$ of $B^{\tau}_{m}$. If $x$ is bi-computable, we are only able to effectively list the elements of basic open or Lacombe sets. In general we cannot effectively search for $z$ in the complement of a basic open set. Here, the topology $\sigma$ comes into play. If $\tau$ is effectively regular with respect to $\sigma$, we can compute a basic open set $B^{\tau}_{s(i,m)}$ and a Lacombe set $L^{\sigma}_{t(i,m)}$ so that  $B^{\tau}_{s(i,m)}$ is a subset of $B^{\tau}_{m}$ containing $x_{i}$, $\overline{B^{\tau}_{m}}$ is contained in $L^{\sigma}_{t(i,m)}$, and $B^{\tau}_{s(i,m)}$ and $L^{\sigma}_{t(i,m)}$ are disjoint. Because we can effectively list the elements of $L^{\sigma}_{t(i,m)}$, we can find a $z$ which is at least in the larger set $L^{\sigma}_{t(i,m)}$ but which is not in the open subset $B^{\tau}_{s(i,m)}$ of $B^{\tau}_{m}$. As will turn out, this is sufficient for our purposes.

\section{Computable quasi-metric spaces}\label{sec-compqm}

\sloppy As seen, quasi-pseudo-metric spaces are particular bi-topological spaces. Let us now study this relationship in the setting of effective spaces. By definition effective spaces satisfy the $T_{0}$ condition. Therefore, we restrict ourselves to the case of quasi-metric spaces.

\begin{definition}\label{dn-compqmetlow}
Let  $(X, \delta)$ be a quasi-metric space and $\QQQ$ a countable dense subset with respect to $\tau_{\delta^{*}}$. Moreover, let $\fun{\beta}{\NN}{\QQQ}$ be a total numbering of $\QQQ$. $(X, \delta)$ is \emph{computable with respect to the lower topology on $[0, \infty)$} (for short: \emph{lower computable}) if the set 
\[
\set{\pair{a, b, c,  e}}{\delta(\beta_{a}, \beta_{b}) < c \cdot 2^{-e}}
\]
is computably enumerable.
\end{definition}

Now, assume that $(X, \delta)$ is lower computable. For $i, m \in \NN$, set $B^{\tau_{\delta}}_{\pair{i,m}} = \ball{\delta}{\beta_{i}}{2^{-m}}$, and for $i, j, m, n \in  \NN$ define
\[
\pair{i, m} \prec_{B^{\tau_{\delta}}} \pair{j, n} \Longleftrightarrow \delta(\beta_{j}, \beta_{i}) + 2^{-m} < 2^{-n}.
\]
Then $\prec_{B^{\tau_{\delta}}}$ is a computably enumerable strong inclusion relation with respect to which the collection of all sets $B^{\tau_{\delta}}_{\pair{i,m}}$ with $i, m \in \NN$ is a strong basis of the quasi-metric topology $\tau_{\delta}$. 

We say that $y \in X$ is \emph{weakly computable} if the set 
\[
\set{\pair{a, e}}{\delta(\beta_{a}, y) < 2^{-e}}
\]
 is computably enumerable, and denote the set of all such points by $X_{wc}$. Let $x^{wc}_{i} = y$, if 
 \[
 W_{i} = \set{\pair{a, e}}{\delta(\beta_{a}, y) < 2^{-e}}.
 \]
Otherwise, let $x^{wc}$ be undefined. Then $x^{wc}$ is a numbering of $X_{wc}$.

\begin{lemma} 
Numbering $x^{wc}$ is acceptable and space $(X_{wc}, \tau_{\delta})$ effective.
\end{lemma}
\begin{proof}
Let $L = \set{\pair{i, \pair{a, e}}}{\pair{a, e} \in W_{i}}$. Then $L$ is computably enumerable and we have for $i \in \dom(x^{wc})$ and $a, e \in \NN$ that
\[
x^{wc}_{i} \in \ball{\delta}{\beta_{a}}{2^{-e}} \Leftrightarrow \pair{a, e} \in W_{i} \Leftrightarrow \pair{i, \pair{a, e}} \in L.
\] 
Hence, $x^{wc}$ is computable from which we obtain with Lemma~\ref{T_eff} that $(X_{wc}, \tau_{\delta})$ is effective.

Next, note that $\set{\pair{a, b, e}}{(\exists n)\, \varphi_{a}(n)\conv \wedge \varphi_{a}(n) \prec_{\tau_{\delta}} \pair{b, e}}$ is computably enumerable. Hence, by the s-m-n theorem \cite{ro:67}, there is some $h \in R^{(1)}$  with
\[
W_{h(a)} = \set{\pair{b, e}}{(\exists n)\, \varphi_{a}(n)\conv \wedge \varphi_{a}(n) \prec_{\tau_{\delta}} \pair{b, e}}.
\]
Now, let $\varphi_{a} \in R^{(1)}$ be decreasing with respect to $ \prec_{\tau_{\delta}}$ so that the collection of all $B^{\tau_{\delta}}_{\varphi_{a}(n)}$ with $n \in \NN$ is a strong basis of the neighbourhood filter of some point $y \in X^{wc}$. Then
\[
y \in \ball{\delta}{\beta_{b}}{2^{-e}} \Leftrightarrow (\exists n)\, \varphi_{a}(n) \prec_{\tau_{\delta}} \pair{b, e} \Leftrightarrow \pair{b, e} \in W_{h(a)},
\]
which means that $x^{wc}_{h(a)} = y$. Thus, $x^{wc}$ allows effective limit passing.
\end{proof}

We call a point $y \in X$ \emph{computable} if the set
\[
\set{\pair{a, b, e}}{\delta(\beta_{a}, y), \delta(y, \beta_{b}) < 2^{-e}}
\]
 is computably enumerable, and denote the set of all computable points by $X_{c}$. Let $x^{c}_{i} = y$, if 
 \[
 W_{i} = \set{\pair{a, b, e}}{\delta(\beta_{a}, y), \delta(y, \beta_{b}) < 2^{-e}}.
 \]
Otherwise, let $x^{c}$ be undefined. Then $x^{c}$ is a numbering of $X_{c}$.

\begin{lemma} 
Numbering $x^{c}$ is bi-acceptable and space $(X_{c}, \tau_{\delta}, \tau_{\delta^{c}})$ bi-effective. Moreover, $\QQQ \subseteq X_{c} \subseteq X_{wc}$. We even have that $\beta \le x^{c} \le x^{wc}$.
\end{lemma}
\begin{proof}
Let $L^{*} = \set{\pair{i, \pair{a, e}}}{(\exists b)\, \pair{a, b, e} \in W_{i}}$. Then $L^{*}$ is computably enumerable. Since $\QQQ$ is dense with respect to $\tau_{\delta}$ as well as $\tau_{\delta^{c}}$ we have for $i \in \dom(x^{c})$ and $a, e \in \NN$ that
\[
x^{c}_{i} \in \ball{\delta}{\beta_{a}}{2^{-e}} \Leftrightarrow (\exists b)\, \pair{a, b, e} \in W_{i} \Leftrightarrow \pair{i, \pair{a, e}} \in L^{*}.
\] 
Hence, $x^{c}$ is computable with respect to $\tau_{\delta}$, from which we obtain with Lemma~\ref{T_eff} that $(X_{c}, \tau_{\delta})$ is effective. Similarly, it follows that $x^{c}$ is also computable with respect to $\tau_{\delta^{c}}$. By Lemma~\ref{T_eff} we therefore have that $(X_{c}, \tau_{\delta^{c}})$ is effective as well. Hence, $(X_{c}, \tau_{\delta}, \tau_{\delta_{c}})$ is bi-effective. 

Next, let $h^{*} \in R^{(1)}$ with
\begin{multline*}
W_{h^{*}(\pair{a_{1}, a_{2}})} = \{\, \pair{b, c, e} \mid (\exists n_{1}) (\varphi_{a_{1}}(n_{1})\conv \wedge \varphi_{a_{1}}(n_{1}) \prec_{\tau_{\delta}} \pair{b, e}) \wedge \mbox{} \\
(\exists n_{2}) (\varphi_{a_{2}}(n_{2})\conv \wedge \varphi_{a_{2}}(n_{2}) \prec_{\tau_{\delta^{c}}} \pair{c, e}) \,\}.
\end{multline*}
Then, if $\varphi_{a_{1}} \in R^{(1)}$ is decreasing with respect to $ \prec_{\tau_{\delta}}$ so that the collection of all $B^{\tau_{\delta}}_{\varphi_{a_{1}}(m)}$ with $m \in \NN$ is a strong basis of the neighbourhood filter of some point $y \in X_{c}$ in topology $\tau_{\delta}$ and $\varphi_{a_{2}} \in R^{(1)}$ is decreasing with respect to $ \prec_{\tau_{\delta^{c}}}$ so that the collection of all $B^{\tau_{\delta^{c}}}_{\varphi_{a_{2}}(m)}$ with $m \in \NN$ is a strong basis of the neighbourhood filter of the same point $y \in X_{c}$ in topology $\tau_{\delta^{c}}$, 
we have
\begin{align*}
&y \in \ball{\delta}{\beta_{b}}{2^{-e}} \cap \ball{\delta^{c}}{\beta_{c}}{2^{-e}} \\
&\Leftrightarrow (\exists n_{1}) (\varphi_{a_{1}}(n_{1}) \prec_{\tau_{\delta}} \pair{b, e}) \wedge (\exists n_{2}) (\varphi_{a_{2}}(n_{2}) \prec_{\tau_{\delta^{c}}} \pair{c, e}) \\
&\Leftrightarrow \pair{b, c, e} \in W_{h^{*}(\pair{a_{1}, a_{2}})},
\end{align*}
from which we obtain that $x^{c}_{h^{*}(\pair{a_{1},a_{2}})} = y$, showing that $x^{c}$ effectively allows bi-limit passing.

Finally, let $f \in R^{(1)}$ with
\[
W_{f(i)} =  \set{\pair{a, e}}{(\exists b)\, \pair{a, b, e} \in W_{i}}.
\]
Then it follows for $i \in \dom(x^{c})$,
\[
\pair{a,e} \in W_{f(i)} \Longleftrightarrow 
(\exists b)\, \delta(\beta_{a}, x^{c}_{i}), \delta(x^{c}_{i}, \beta_{b}) < 2^{-e} \Longleftrightarrow \delta(\beta_{a}, x^{c}_{i}) < 2^{-e},
\]
where the last equivalence holds as $\QQQ$ is dense with respect to both topologies, $\tau_{\delta}$ and $\tau_{\delta^{c}}$. Consequently, $x^{wc}_{f(i)} = x^{c}_{i}$. That is, $x_{c} \le x^{wc}$.

Since $(X, \delta)$ is lower computable, there is some $g \in R^{(1)}$ with
\[
W_{g(b)} = \set{\pair{a, d, e}}{\delta(\beta_{a}, \beta_{b}), \delta(\beta_{b}, \beta_{d}) < 2^{-e}}.
\]
Then $\beta_{b} = x^{c}_{g(b)}$. That is $\beta \le x^{c}$.
\end{proof}
In what follows we will  always index the computable elements of a lower computable quasi-metric space by using  $x^{c}$, or an equivalent numbering.

Our next goal is to show that $\tau_{\delta}$ is effectively regular with respect to $\tau_{\delta^{c}}$. Let to this end $i \in \dom(x^{c})$ and $a, m \in \NN$ so that $x^{c}_{i} \in B^{\tau_{\delta}}_{\pair{a,m}}$. We have to effectively find some $b, n \in \NN$ so that $\pair{b, n} \prec_{B^{\tau_{\delta}}} \pair{a,m}$ and $x^{c}_{i} \in B^{\tau_{\delta}}_{\pair{b,n}}$. As we have seen in the proof of Lemma~\ref{lem-qmetreg} such numbers do exist.

Let $(X, \delta)$ be lower computable and $L$ witness the computability of $x^{c}$ with respect to $\tau_{\delta}$. Then
\[
E = \set{\pair{i, a, m, b, n}}{\pair{b, n} \prec_{B^{\tau_{\delta}}} \pair{a,m} \wedge \pair{i, \pair{b,n}} \in L}
\]
is computably enumerable. For $i, a, m \in \NN$ let $i', a', m', b', n'$ be the first element in some fixed enumeration of $E$ with $i' = i$, $a' = a$ and $m' = m$. Define $s(i, \pair{a, m}) = \pair{b', n'}$. Then $s$ is computable and for $i \in \dom(x)$ and $a, m \in  \NN$  with $x_{i} \in B^{\tau_{\delta}}_{\pair{b,n}}$ we have that $s(i,\pair{a,m})\conv$ and $x_{i} \in B^{\tau_{\delta}}_{s(i, \pair{a,m})} \subseteq B^{\tau_{\delta}}_{\pair{a, m}}$.

Now, in order to proceed as in the proof of Lemma~\ref{lem-qmetreg}, we need that also the set $\set{\pair{a, b, c, e}}{\delta(\beta_{a}, \beta_{b}) >  c \cdot 2^{-e}}$ is computably enumerable.

\begin{definition}\label{dn-compqmet}
Let  $(X, \delta)$ be a quasi-metric space and $\QQQ$ a countable dense subset with respect to $\tau_{\delta^{*}}$. Moreover, let $\fun{\beta}{\NN}{\QQQ}$ be a total numbering of $\QQQ$. $(X, \delta)$ is \emph{computable} if it is lower computable and, in addition, the set
\[
\set{\pair{a, b, c, e}}{\delta(\beta_{a}, \beta_{b}) > c \cdot 2^{-e}}
\]
is computably enumerable.
\end{definition}

\begin{example}\label{ex-compreal}
Let $\QQQ$ be the set of all dyadic rationals and set $\beta_{\pair{a,b,c,e}} = (a - b) \cdot 2^{c - e}$. Then $\QQQ$ is dense in $\RR_{c}$ with respect to $\LLL \vee \UUU$. Since the order on the rational numbers is decidable, it follows that both $(\RR_{c}, \delta_{L})$ and $(\RR_{c}, \delta_{U})$ are computable. 
\end{example}

So, assume that $(X, \delta)$ is even computable. As we have seen in the preceding lemma,  $x^{c}$ is bi-computable. Then it follows that the set
\[
A = \{\, \pair{i, a, m, b, c} \mid \delta(\beta_{a}, \beta_{b}) > 2^{-\pi_{2}(s(i, \pair{a,m})} \wedge 
\delta(\beta_{a}, \beta_{b}) > 2^{-(c-1)} + 2^{-\pi_{2}(s(i, \pair{a,m})} \,\}
\]
is computably enumerable. Hence, there is a computable function $t \in R^{(2)}$ with 
\[
W_{t(i, \pair{a, m})} = \set{\pair{b, c}}{\pair{i,a,m,b,c} \in A}.
\]
 As in the proof of Lemma~\ref{lem-qmetreg} it follows for every $\pair{b, c} \in W_{t(i, \pair{a, m})}$ that $\ball{\delta^{c}}{\beta_{b}}{2^{-c}}$ and $B^{\tau_{\delta}}_{s(i,\pair{a, m})}$ are disjoint. Moreover, the complement of $B^{\tau_{\delta}}_{\pair{a,m}}$ is contained in $\bigcup \{\, \ball{\delta^{c}}{\beta_{b}}{2^{-c}} \mid \pair{b,c} \in W_{t(i, \pair{a, m})} \,\}$. By definition the latter set is $L^{\sigma}_{t(i, \pair{a,m})}$. 

This shows that $\tau_{\delta}$ is effectively regular with respect to $\tau_{\delta^{c}}$. In the same way it follows that  $\tau_{\delta^{c}}$ is effectively regular with respect to $\tau_{\delta}$. 
\begin{theorem}\label{thm-effqmetpairreg}\sloppy
Let $(X, \delta)$ be a quasi-metric space and $\QQQ$ a countable subset that is dense with respect to $\tau_{\delta^{*}}$ so that $(X, \delta)$ is computable.  Then $(X_{c}, \tau_{\delta}, \tau_{\delta^{c}})$ is an effectively pairwise regular bi-effective bi-topological space.
\end{theorem}

\section{The continuity theorem}\label{sec-cont}

As already said, in Markov-style constructive mathematics objects are constructed by algorithmic instructions. So, in order to define a function on such objects one has to say how it algorithmically transforms the algorithms underlying the objects. Formalised algorithms are pieces of text and can therefore coded by natural numbers.

Let ${\TTT}^\prime = (T^\prime, \tau^\prime)$ be a further countable topological space with a countable basis ${\BBB}^\prime$, and let $x^\prime$ and $B^\prime$, respectively, be numberings of $T^\prime$ and ${\BBB}^\prime$. 
\begin{definition}\label{dn-markovconst}
A map $F \colon T \rightarrow T^\prime$ is an {\em effective operator}, if there is a function $f \in P^{(1)}$ such that $f(i)\mathclose\downarrow\in \dom(x^\prime)$ and $F(x_i) =
x^\prime_{f(i)}$, for all $i \in \dom(x)$.
\end{definition}

The \emph{continuity problem} is the question whether effective operators are (effectively) continuous.

\begin{definition}\label{dn-effcont}
A map $\fun{F}{T}{T'}$ is said to be
\begin{enumerate}
\item\label{dn-effcont-1}
{\em effectively pointwise continuous}, if there is a function $h \in P^{(2)}$ such that
for all $i \in \dom(x)$ and $n\in \dom(B^\prime)$ with 
$F(x)_i\in B_n^\prime$, $h(i, n)\mathclose\downarrow\in \dom(B)$, 
$x_i\in B_{h(i, n)}$, and $F[B_{h(i, n)}] \subseteq B_n^\prime$

\item\label{dn-effcont-2} 
{\em effectively continuous\/}, if there is a
function $g \in P^{(1)}$ such that for all $n \in \dom(B^\prime)$,
$g(n)\mathclose\downarrow \in \dom(L^\tau)$ and $F^{-1}[B_n^\prime] =
L_{g(n)}^\tau$.
\end{enumerate}
\end{definition}

\begin{lemma}[\cite{sp}]\label{compar_cont}\hfill
\begin{enumerate}
\item\label{compar_cont_1} Let $x$ be computable. If $F$ is effectively
          continuous, then it is also effectively pointwise continuous.
\item\label{compar_cont_2} Let $\TTT$ be effective and computably
          separable, $x$ be acceptable, $x^\prime$ be computable, and
          $F$ be effective. Then, if $F$ is effectively pointwise
          continuous, it is also effectively continuous.
\end{enumerate}
\end{lemma}

In \cite{sp,sy:84} the effective pointwise continuity of effective operators has been studied in a general context. One of the main requirements $F$ must fullfill is that of having a witness for non-inclusion.
\begin{definition}\label{dn-witnonincl}
Suppose that $\fun{F}{T}{T'}$. Then $F$ is said to have a {\em witness for non-inclusion} if there exist
functions $s\in P^{(2)}$ and $r \in P^{(3)}$ such that for all $i \in
\dom(x)$, $n\in \dom(B)$ and $m \in \dom(B^\prime)$ the following hold:
\begin{enumerate}
\item\label{dn-witnonincl-1} If $F(x_i) \in B_m^\prime$, then $s(i ,m)
              \mathclose\downarrow\in \dom(M)$ and $x_i \in
              M_{s(i, m)} \subseteq F^{-1}[B_m^\prime]$.
\item\label{dn-witnonincl-2} If, in addition, $F[B_n] \not\subseteq
              B_m^\prime$, then also $r(i, n, m)\mathclose\downarrow\in
              \dom(x)$ and $x_{r(i, n, m)} \in B_n \setminus
              M_{s(i, m)}$.
\end{enumerate}
\end{definition}

\begin{theorem}[\cite{sp}]\label{general-cont}
Let $\TTT$ be effective, $x$ be acceptable and $x^\prime$ be
computable. Moreover, let $F: T \rightarrow T^\prime$ be an effective operator 
that has a witness for non-inclusion. Then $F$ must be
effectively pointwise continuous. If, in addition, $\TTT$ is
computably separable, then $F$ is even effectively continuous. 
\end{theorem}

As was shown in addition, for computably separable semi-regular effective spaces the requirement that $F$ has a witness for non-inclusion  is also a necessary condition for  $F$ to be effectively pointwise continuous.

\begin{proposition}[\cite{sp}]\label{cont-eff}
Let $\TTT'$ be an effective $T_0$-space such that the strong
inclusion relation $\prec_{B'}$ is computably enumerable. Moreover, let $x$ be
computable and $x'$ allow effective limit passing.  Then every
effectively continuous map $\fun{F}{T}{T'}$ is an effective operator.
\end{proposition}

In what follows we will show that effective operators $\fun{F}{T}{T'}$ do have a witness for non-inclusion, if $\TTT' = (T', \tau', \sigma')$ is a countable second-countable bi-topological space so that $\tau'$ is effectively regular with respect to $\sigma'$. Assume to this end that $\BBB^{\tau'}$ and $\BBB^{\sigma'}$, respectively, are countable bases of $\tau'$ and $\sigma'$ with indexings $B^{\tau'}$ and/or $B^{\sigma'}$.

\begin{proposition}\label{prop-wit}
Let $\TTT = (T, \tau)$ be effective and computably separable with acceptable numbering $x$ of $T$. Moreover, let $\TTT' = (T', \tau', \sigma')$ be a countable second-countable bi-topological space and $x'$ a bi-computable numbering of $T'$. Finally, let $\fun{F}{T}{T'}$ be an effective operator. 
If $\tau'$ is effectively regular with respect to $\sigma'$, then $F$ has a witness for non-inclusion.
\end{proposition}
\begin{proof}
Let $k \in R^{(1)}$ so that $\QQQ = \set{x_{k(a)}}{a \in \NN}$ is dense in $\TTT$. Moreover, let $L, L' \subseteq \NN$, $s', t' \in P^{(2)}$, and $f \in P^{(1)}$, respectively, witness that $x$ is computable, $x'$ is computable with respect to $\sigma'$ and $\tau'$ is effectively regular with respect to $\sigma'$, and also $F$ is an effective operator. Then
\[
A = \set{\pair{i,m,n,a}}{\pair{k(a), n} \in L \wedge (\exists \ell \in W_{t'(i,n)})\, \pair{f(k(a)), \ell} \in L'}
\]
is computably enumerable. For $i, m, n \in \NN$ let $\pair{i',m',n',a'}$ be the first enumerated  element in a fixed enumeration of $A$ with $i' = i$, $m' = m$, and $n' = n$. Define $r(i,m,n) = k(a')$. Then $r \in P^{(3)}$.

Now, assume that $F(x_{i}) \in B^{\tau'}_{m}$. Then $s'(i,m)\conv \in \dom(B^{\tau'})$ and $F(x_{i}) \in B^{\tau'}_{s'(i,m)} \subseteq B^{\tau'}_{m}$. If, in addition, $F[B_{n}] \not\subseteq B^{\tau'}_{m}$ then $F[B^{\tau'}_{n}]$ hits $L^{\sigma'}_{t'(i,m)}$, i.e., $B_{n}$ hits $F^{-1}[L^{\sigma'}_{t'(i,m)}]$. Since the latter set is completely enumerable, it follows with Lemma~\ref{intersec-DB} that $B_{n}$ hits $F^{-1}[L^{\sigma'}_{t'(i,m)}] \cap \QQQ$. Thus, there is some $a \in \NN$ so that $x_{k(a)} \in B_{n}$ and $F(x_{k(a)}) \in L^{\sigma'}_{t'(i,m)}$, that is $F(x_{k(a)}) \notin B^{\tau'}_{s'(i,m)}$, which shows that $r(i,m,n)\conv \in \dom(x)$ and $x_{r(i,m,n)} \in B_{n} \setminus F^{-1}[B^{\tau'}_{s'(i,m)}]$.

Finally, let $L'' \subseteq \NN$ witness that $x'$ is computable with respect to $\tau'$. Then there is some $s \in R^{(2)}$ with $W_{s(i,m)} = \set{j \in \NN}{\pair{f(j), s'(i,m)} \in L''}$, which means that $M_{s(i,m)} = F^{-1}[B^{\tau'}_{s'(i,m)}]$.  
\end{proof}

As a consequence of the General Continuity Theorem~\ref{general-cont} we obtain the following answer to the continuity problem for bi-topological spaces.
\begin{theorem}\label{thm-contbitop}
Let $\TTT = (T, \tau)$ be effective and computably separable with acceptable numbering of $T$. Moreover, let $\TTT' = (T', \tau', \sigma')$ be a countable second-countable bi-topological space with a bi-computable numbering of $T'$. If $\tau'$ is effectively regular with respect to $\sigma'$, then every effective operator $\fun{F}{T}{T'}$ is effectively continuous with respect to $\tau$ and $\tau'$.
\end{theorem}

Let us next study some consequences of the above theorems.

\begin{definition}
 Let $\TTT = (T, \tau, \sigma)$ and $\TTT' = (T', \tau', \sigma')$, respectively, be countable second-countable bi-topological spaces with indexings $B$ and $B'$ of their bases and numberings $x, x'$ of $T$ and $T'$. A map $\fun{F}{T}{T'}$ is \emph{effectively bi-continuous} if $F$ is effectively continuous with respect to $\tau$ and $\tau'$ as well as with respect to $\sigma$ and $\sigma'$.
 \end{definition}
 
\begin{theorem}
Let $\TTT = (T, \tau, \sigma)$ be a countable second-countable bi-topological space with bi-computable numbering $x$. Moreover, let $\TTT' = (T', \tau', \sigma')$ be bi-effective such that $\prec_{B^{\tau'}}$ as well as $\prec_{B^{\sigma'}}$ are computably enumerable, and $x'$ be a numbering of $T'$ that allows effective bi-limit passing. Then every effectively bi-continuous map $\fun{F}{T}{T'}$  is an effective operator. 
\end{theorem}
The proof is a direct translation of the proof of \cite[Proposition 7.9]{sp} to the bi-topological case.
 
 If $T$ contains a countable set $Z$ that is dense with respect to both topologies, $\tau$ and $\sigma$ then $\TTT = (T, \tau, \sigma)$ is called \emph{bi-separable}. In case that $Z$ is enumerable with respect to the given numbering of $T$, we say that $\TTT$ is \emph{computably bi-separable}.
 
\begin{theorem}\label{thm-marcont}
Let $\TTT = (T, \tau, \sigma)$ be bi-effective and computably bi-separable with bi-acceptable numbering $x$. Moreover, let $\TTT' = (T', \tau', \sigma')$ be a countable second-countable bi-topological space with bi-computable numbering $x'$ so that $\TTT'$ is effectively pairwise regular. Then every effective operator $\fun{F}{T}{T'}$ is effectively bi-continuous.
\end{theorem}
 
\begin{corollary}\label{cor-marmcont}
Let $\TTT = (T, \tau, \sigma)$ and $\TTT' = (T', \tau', \sigma')$  be bi-effective with bi-acceptable numberings so that $\TTT$ is computably bi-separable and $\TTT'$ effectively pairwise regular. Moreover, let $\prec_{B^{\tau'}}$ and $\prec_{B^{\sigma'}}$ be computably enumerable. Then every map $\fun{F}{T}{T'}$ is effectively bi-continuous if and only if $F$ is an effective operator.
\end{corollary}

As we have seen in the previous section, under rather natural effectivity requirements quasi-metric spaces are effectively pairwise regular. 

\begin{theorem}\label{thm-effqmcont}
For $\nu = 1,2$, let $(X^{(\nu)}, \delta_{\nu})$ be a bi-separable quasi-metric space so that $(X^{(1)}, \delta_{1})$ is lower computable and $(X^{(2)}, \delta_{2})$ is computable. Then every map $\fun{F}{X_{c}^{(1)}}{X_{c}^{(2)}}$  is an effective operator if and only if $F$ is effectively bi-continuous.
\end{theorem}

This result extends the classical theorems of Moschovakis \cite{mosch63,mosch64} and Ce\u{\i}tin~\cite{ce} for recursive metric spaces to the quasi-metric case. In Theory of Computing the use of quasi-metric spaces has particularly been propagated by M.B.~Smyth~\cite{sm:88}, since they comprise metric spaces as well as domains, both used for giving a precise meaning to programming language constructs.

\section{Friedberg's counterexample}\label{sec-fried}

Friedberg~\cite{fr} has presented an effective operator $\fun{F}{M}{S}$ that is not continuous. Here, $M$ is the set of computable elements of a computable metric space and $\SSS = (S, \sigma)$ with $S= \{ \bot, \top \}$ and $\sigma = \{\emptyset, \{ \top \}, S \}$ is  \emph{Sierpinski space}. Define
\[
\delta(y, z) = \begin{cases}
                           0 & y = z \vee (y = \bot \wedge z = \top) \\
                           1 & \text{otherwise}.
                			\end{cases}
\]
Then $\delta$ is a quasi-metric on $S$ that generates the given topology $\sigma$. $S$ is dense in itself with respect to $\tau_{\delta}$ as well as $\tau_{\delta^{c}}$. Let 
\[
\beta_{i} = \begin{cases}
			\bot & i = 0 \\
			\top & \text{otherwise}.
			\end{cases}
\] 
be an indexing of the dense base.	
As is easily verified, the quasi-metric space $(S, \delta)$ is  computable. Moreover, the elements of $S$ are computable as well. Note further that
\[
B^{\delta}_{\pair{a,e}} = \ball{\delta}{\beta_{a}}{ 2^{-e}} = \begin{cases}
                                                                                       S & \text{if $a = 0$,} \\
                                                                                       \{ \top \} & \text{otherwise.}
                                                                          \end{cases}
 \]
 Hence, $\pair{a,e} \prec_{B^{\delta}} \pair{a',e'}$, exactly if $a \not= 0$ and $a' = 0$. It follows that  $\prec_{B^{\delta}}$ is computably enumerable.                                                                      

In Friedberg's example, $M$ is indexed by an acceptable numbering. Moreover, $S$ is indexed by a numbering $x^{S}$ that is acceptable with respect to $\sigma$, i.e.\ $\tau_{\delta}$. 
Then the set $\{ \bot \}$ is not completely enumerable. Otherwise, it would be upwards closed under $\prec_{\sigma}$, by Proposition~\ref{enum_spec}. Since $\prec_{\sigma} = \prec_{\tau_{\delta}} = \sqsubseteq_{\delta}$, it would follow that $\top \in \{ \bot \}$, which is not the case. 

Assume now that $x^{S}$ is also computable with respect to $\tau_{\delta^{c}}$. As $\{ \bot \}$ is basic open with respect to $\tau_{\delta^{c}}$, it follows that $\{ \bot \}$ is completely enumerable with respect to $x^{S}$, a contradiction.

\begin{proposition}
In Theorem~\ref{thm-marcont} we cannot dispense with the requirement that numbering $x'$ is also computable with respect to topology $\sigma'$.
\end{proposition}

\section{Conclusion}\label{sec-conc}

Bi-topological spaces are not just an extension of topological spaces by adding a further topology on the underlying set. Many important notions and results naturally extend to the bi-topological case. Moreover, spaces of this kind have turned out to allow for interesting representations of symmetric (i.e.\ Hausdorff) topologies by two asymmetric ones.

In the present paper the framework of effective spaces developed in~\cite{sp} was extended to the bi-topological case and the question of whether effective operators are effectively continuous was considered. It was shown that under the effectivity conditions that are always satisfied in applications, an effective operator $\fun{F}{\TTT}{\TTT'}$, where $\TTT = (T, \tau)$ is an effective topological space and $\TTT' = (T', \tau', \sigma')$ a bi-topological space, is effectively continuous with respect to $\tau$ and $\tau'$, if $\tau'$ is effectively regular with respect to $\sigma'$. 

Interestingly, it turned out that though the continuity is stated for $F$ as a map from $(T, \tau)$ to $(T', \tau')$ and topology $\sigma'$ is more or less considered as an auxiliary topology, one nevertheless has to require each of the general effectivity conditions simultaneously for both topologies, $\tau'$ and $\sigma'$. It thus follows that if also $\TTT = (T, \tau, \sigma)$ is a bi-topological space and $\TTT'$ is effectively pairwise regular, then $F$ is even effectively bi-continuous. Conversely, it was shown that effectively bi-continuous operators are also effective, independently of any regularity requirement.

As an important example, computable quasi-metric spaces were studied. The topologies induced by the given quasi-metric and its conjugate are considered in this case. All general effectivity requirements that had to be made in the general study are now satisfied. Moreover, such spaces are effectively pairwise regular.

\bibliographystyle{plain}

\end{document}